\documentclass[11pt,reqno]{amsart}
\usepackage[top=1.2in, bottom=1.2in, left=1.2in, right=1.2in]{geometry}
\usepackage{amsfonts, amssymb, amsmath, mathtools, dsfont}
\usepackage[linktoc=page, colorlinks, linkcolor=blue, citecolor=blue]{hyperref}
\usepackage{enumerate}

\usepackage{xcolor}

\usepackage{graphicx}

\usepackage{subcaption} 
\usepackage[font=small]{caption} 
\numberwithin{equation}{section}

\DeclareMathOperator{\E}{\mathbb{E}} 
\DeclareMathOperator{\supp}{supp} 
\DeclareMathOperator{\tr}{tr} 

\DeclarePairedDelimiter \abs{\lvert}{\rvert} 
\DeclarePairedDelimiter \norm{\lVert}{\rVert} 
\DeclarePairedDelimiterX \ip[2]{\langle}{\rangle}{#1,#2} 
\DeclarePairedDelimiterXPP \Prob[1]{\mathbb{P}}\{\}{}{
   
   #1} 
\DeclarePairedDelimiterXPP \Probevent[1]{\mathbb{P}}(){}{
   
   #1} 
 
\def \R {\mathbb{R}}

\def \I {\mathcal{I}}

\def \e {\varepsilon}

\def \l {\lambda}

\def \tran {\mathsf{T}}

\newtheorem{theorem}{Theorem}[section]
\newtheorem{proposition}[theorem]{Proposition}
\newtheorem{corollary}[theorem]{Corollary}
\newtheorem{lemma}[theorem]{Lemma}

\newtheorem{definition}[theorem]{Definition}

\theoremstyle{remark}

\usepackage{dirtytalk}

\title{Discrepancy and Fisher information}

\author{Gleb Smirnov}
\address{Mathematical Sciences Institute, 
Australian National University, Canberra, Australia}
\email{gleb.smirnov@anu.edu.au}

\author{Roman Vershynin}
\address{Department of Mathematics, University of California, Irvine, U.S.A.}
\email{rvershyn@uci.edu}

\thanks{R.V. is partially supported by the NSF Grant DMS 2451011 and the U.S. Air Force Grant FA9550-25-1-0294.
}
 
\begin{document}

\maketitle

\begin{abstract}
We give an online algorithm that keeps a symmetric random walk inside a convex body by discarding some of its steps. The expected number of discarded steps is controlled by a Fisher-information-type quantity associated with the body. For the cube, this gives a dimension-free bound: a walk with unit Euclidean steps can be kept bounded in all coordinates while discarding only a small constant fraction of the steps on average.
\end{abstract}

\section{Introduction}

\subsection{Random walks}

Consider a symmetric random walk in $\R^d$ with given steps:
$$
S_k = \sum_{j=1}^k \e_j v_j, \quad 
k=1,2, \ldots
$$
where $\e_j = \pm 1$ are independent Rademacher random variables and $v_j \in \R^d$ are given vectors. If the steps $v_j$ have unit length, the random walk drifts arbitrarily far from the origin over time. In particular, it exits any given compact set $K \subset \R^d$ in finite time almost surely.

{\em At what cost can one keep the random walk inside $K$ indefinitely? In particular, can this be achieved by discarding only a small fraction of the steps?}

\subsection{Fisher information}

A possible answer can be formulated in terms of the information content of $K$. A fundamental notion in statistics is the Fisher information. For a family of probability densities $p_\theta(\cdot)$ on $\R^d$, parametrized by a vector $\theta \in \R^m$, the {\em Fisher information matrix} is a symmetric, positive-semidefinite $m \times m$ matrix defined by
$$
\I(\theta)
\coloneqq \E \big( \nabla_\theta \log p_\theta(X) \big) \big( \nabla_\theta \log p_\theta(X) \big)^\tran, 
\quad \text{where} \quad  X \sim p_\theta.
$$
The random vector $\nabla_\theta \log p_\theta(X)$ is called the {\em score}. 
It measures the sensitivity of the log-likelihood $\log p_\theta(X)$ to 
infinitesimal changes in the parameter $\theta$. Under the usual regularity 
assumptions, the score has mean zero. Hence $\I(\theta)$ is the covariance 
matrix of the score.

Thus, for any given unit vector $v \in \R^m$, the quadratic form 
$$
v^\tran \I(\theta) v
$$ 
is the variance of the score in the direction of $v$. It quantifies how much information the data carries about the parameter $\theta$ along that direction.

As an example, let's fix one probability density $\pi(\cdot)$ and consider the family of all translates: 
$$
p_\theta(x) = \pi(x+\theta), \quad \theta \in \R^d.
$$
Then the Fisher information matrix, evaluated at the origin, is
\begin{equation}	\label{eq: Ipi}
	\I_\pi 
	\coloneqq \I(0)
	= \E \big( \nabla \log \pi(X) \big) \big( \nabla \log \pi(X) \big)^\tran, 
	\quad \text{where} \quad  X \sim \pi.
\end{equation}

\subsection{Discrepancy}

The following is our main result.
\begin{theorem}	 \label{thm: main}
Let $K \subset \R^d$ be a convex, centrally symmetric set with nonempty interior. Let $v_1,\dots,v_n \in \R^d$ be fixed vectors, and let $\e_1,\ldots, \e_n$ be independent Rademacher random variables. Let $\pi(\cdot)$ be a probability density supported in the interior of $K$. Then there exists an online algorithm that discards some steps $\e_j v_j$ such that, for each $k$,
	$$
	\sum_{j \le k \text{ accepted}} \e_j v_j \in 2K
	\quad \text{deterministically},
	$$
	and
	\begin{equation}	\label{eq: Ediscarded}
		\E \#\{\text{discarded}\}
		\le \frac{1}{2} \sum_{j=1}^n \sqrt{v_j^\tran \I_\pi v_j}, 
	\end{equation}
	where $\I_\pi$ denotes the Fisher information matrix defined in \eqref{eq: Ipi}.
\end{theorem}

The proof will be given in Section~\ref{s: main theorem proof} and is based on the Metropolis algorithm with stationary distribution $\pi$. Intuitively, when $v_j^\tran \I_\pi v_j$ is small, the density $\pi$ is only weakly affected by the next step $\e_j v_j$, so this step is unlikely to be rejected by the Metropolis rule.

We will now explain how to choose a suitable \(\pi\) in the theorem, and then specialize it to the case of the cube $K = [-T,T]^d$. 

\subsection{Information capacity}

To get a more handy bound in Theorem~\ref{thm: main}, note that 
$$
v_j^\tran \I_\pi v_j 
\le \norm{\I_\pi} \norm{v_j}_2^2,
$$
where $\norm{\I_\pi}$ denotes the operator norm of $\I_\pi$. Now we can optimize the density $\pi$, and get: 
$$
\E \#\{\text{discarded}\}
\le \frac{1}{2} I(K) \sum_{j=1}^n \norm{v_j}_2,
$$ 
where $I(K)$ is the ``information capacity'' of $K$, defined as
$$
I(K) \coloneqq \inf_{\supp(\pi) \subset K} \norm{\I_\pi}^{\frac12}.
$$
The infimum is over smooth probability densities $\pi$ supported in the interior of $K$.

\subsection{Dirichlet eigenvalues}

It may not be easy to compute (approximately or exactly) the information content 
$I(K)$ for a given domain $K$. The operator norm is not an easy quantity to handle, and it is not clear how to find a density $\pi$ that makes it small. To get around this problem, let's instead of the operator norm consider a simpler 
quantity -- trace. Assume that
$$
\pi = \psi^2
$$
for some function $\psi$ supported in the interior of $K$. Hence we have
$$
\int_K \psi^2 = \int_K \pi = 1 
\quad \text{and} \quad 
\nabla \log \pi (x) = \frac{2 \nabla \psi(x)}{\psi(x)}.
$$
Substitute $x=X$ to express the Fisher information matrix \eqref{eq: Ipi} as
\begin{equation}	\label{eq: I pi psi}
	\I_\pi =  4 \int_K \nabla \psi(x) \nabla \psi(x)^\tran \, dx.
\end{equation}
Then 
\begin{equation}	\label{eq: tr I pi}
	\tr \I_\pi = 4 \int_K \norm{\nabla \psi(x)}_2^2 \, dx.
\end{equation}
The integral on the right-hand side is the Dirichlet energy of \(\psi\). The first Dirichlet eigenvalue of \(K\) is defined by
\begin{equation} \label{eq: lambda1 var}
\lambda_1(K)
=
\inf \left\{
\int_K \norm{\nabla \psi}_2^2 \, dx :
\psi \in C^\infty,\;
\supp(\psi) \subset \operatorname{int} K,\;
\int_K \psi^2 \, dx = 1
\right\}.
\end{equation}
Assume for now that 
\(K\subset \mathbb{R}^d\) is the closure
of a bounded connected domain with smooth boundary. This smoothness assumption
is sufficient for our application: any compact convex set with nonempty interior can be approximated by a smooth, convex, set, obtained by an arbitrarily small perturbation of \(K\). For bounded, smooth \(K\), this infimum is attained\footnote{The infimum in
\eqref{eq: lambda1 var} is attained in \(W^{1,2}_0(K)\), the Sobolev closure of
\(C^\infty_0(K)\), where \(C^\infty_0(K)\) denotes the space of smooth functions
supported in \(\operatorname{int}K\).} by the first Dirichlet eigenfunction; this eigenfunction satisfies:
\begin{equation} \label{eq: PDE}
	\begin{cases}
		\Delta \psi = -\lambda_1(K)\psi & \text{in the interior of } K,\\
		\psi = 0 & \text{on } \partial K,
	\end{cases}
	\qquad \int_K \psi^2 = 1.
\end{equation}
Using this $\psi$ in \eqref{eq: tr I pi}, we get:
\begin{equation}	\label{eq: tr I lambda}
	\tr \I_\pi = 4 \l_1(K).
\end{equation}

See~\cite{Gilb-Tr} for the classical Dirichlet problem on smooth domains. The 
work~\cite{L-U} covers substantially weaker assumptions; in particular, it applies to convex domains and allows one to bypass the smoothing step above.

\begin{corollary}[Isotropic random walk]	
	Let $K \subset \R^d$ be a compact, convex, centrally symmetric set with nonempty interior. Let $v_1,\dots,v_n$ be independent symmetric isotropic random vectors in $\R^d$. (That is, for each $j$, the random vectors $v_j$ and $-v_j$ have the same distribution, and $\E v_j v_j^\tran = I_d$.) Then there exists an online
    \footnote{Here \say{online} 
    means that, upon observing \(v_j\), 
    one must immediately decide whether to accept or discard it; the decision may depend on the previous steps, but not on future steps, and it cannot be changed later.} algorithm that discards some steps $v_j$ such that, for each $k$,
	\[
		\sum_{j \le k \text{ accepted}} v_j \in 2K
		\quad \text{deterministically},
	\]
	and
	\[
		\E \#\{\text{discarded}\}
		\le \sqrt{\lambda_1(K)} \, n,
	\]
	where $\lambda_1(K)$ is the first Dirichlet eigenvalue of $K$.
\end{corollary}
\begin{proof}
	By symmetry and independence, the random vectors $v_j$ have the same joint distribution as $\e_j v_j$. Apply Theorem~\ref{thm: main} conditionally on $(v_j)$ and for $\pi=\psi^2$, where $\psi$ is the first Dirichlet eigenfunction of $K$. Then, by isotropy and \eqref{eq: tr I lambda}, we have
	$$
	\E \sqrt{v_j^\tran \I_\pi v_j}
	\le \sqrt{\E v_j^\tran \I_\pi v_j}
	= \sqrt{\tr \I_\pi} 
	= 2\sqrt{\lambda_1(K)}. \qedhere
	$$ 
\end{proof}

\subsection{Example: the cube}

Let's specialize to the benchmark domain $K$ in combinatorial discrepancy theory --  the cube, or a ball of the $\ell^\infty$ norm.

\begin{corollary}[Combinatorial discrepancy]\label{cor: upper}
	Let $v_1,\dots,v_n \in \R^d$ be any fixed vectors, and let $\e_1,\ldots, \e_n$ be independent Rademacher random variables. Then there exists an online algorithm that discards some terms $\e_j v_j$ so that, for each $k$,
	\begin{equation}	\label{eq: upper}
		\norm[\Bigg]{\sum_{j \le k \text{ accepted}} \e_j v_j}_\infty \le 2T
		\quad \text{deterministically},
	\end{equation}
	and
	$$
	\E \#\{\text{discarded}\}
	\le \frac{\pi}{2T} \sum_{j=1}^n \norm{v_j}_2.
	$$
\end{corollary}

In particular: {\em if all $v_j$ are unit vectors, then we can keep all the coordinates of the random walk bounded by $32$ for any finite horizon by discarding at most $10\%$ of terms on average.}

\begin{proof}
For the cube $K = [-T,T]^d$, the first Dirichlet eigenvalue and eigenfunction are
$$
\psi(x_1,\dots,x_d)
= T^{-d/2}\prod_{i=1}^d \cos\!\left(\frac{\pi x_i}{2T}\right), 
\quad
\l_1 = \l_1(K) = \frac{\pi^2 d}{4T^2}
$$
Indeed, this pair solves the PDE \eqref{eq: PDE}, and the value $\l_1>0$ above is the smallest for which a solution exists. Now choose the density $\pi = \psi^2$. Its Fisher information matrix $\I_\pi$, defined in \eqref{eq: I pi psi}, has zero off-diagonal entries and equal diagonal entries (by symmetry), so $\I_\pi$ is a multiple of identity. Since $\tr \I_\pi = 4 \l_1(K)$ by \eqref{eq: tr I lambda}, it follows that
$$
\I_\pi = \frac{4 \l_1(K)}{d} I 
= \frac{\pi^2}{T^2} I 
$$
where $I$ denotes the identity matrix in $\R^d$. Plugging into Theorem~\ref{thm: main} completes the proof.
\end{proof}

\subsection{Related results}

Corollary \ref{cor: upper} strengthens our earlier result \cite[Proposition~2]{Sm-Versh-2}, where the same bound (up to a constant) appeared with the $\ell^1$-norms of the vectors $v_j$ in place of their $\ell^2$-norms. As a consequence, the two-sample discrepancy bound $O(\log^{2 d} n)$ in \cite[Theorem~1]{Sm-Versh-2} improves to $O(\log^{3d/2} n)$.

Beginning from the seminal work \cite{banaszczyk}, a substantial body of research has been devoted to discrepancy minimization, see e.g. \cite{AL-Saw, Bans-Spencer, Bans-JSS, Bans-JMSS, Kulk-Reis-Roth, DFGGR, bandeira2022remark}. These works study the {\em existence} of signs $\e_1,\ldots,\e_n \in \{-1,1\}$ for which the partial sums $\varepsilon_1 v_1 + \cdots + \varepsilon_k v_k$ remain small. The novelty of Theorem~\ref{thm: main} is that completely {\em random} signs work after a small number of terms is discarded.

\subsection{Optimality}

For unit vectors $v_1,\ldots,v_n$, Corollary~\ref{cor: upper} says that one can maintain the low discrepancy bound \eqref{eq: upper} by discarding $O(n/T)$ vectors. This discard rate is optimal even in dimension $d=1$. The next result shows that in order to keep the random walk within $[-T,T]$, one must discard at least $\Omega(n/T)$ vectors. 

\begin{proposition}[Optimality]	\label{prop: lower}
	Let $\e_1,\ldots, \e_n$ be independent Rademacher random variables. 
	Consider any discarding algorithm\footnote{Here, an algorithm can be any function that maps a sign string $\e_1,\ldots, \e_n$ to a subset of $\{1,\ldots,n\}$ of the terms to be discarded.} (online or not) that maintains, for each $k$,
	$$
	\abs[\Bigg]{\sum_{j \le k \text{ accepted}} \e_j} \le T
	\quad \text{deterministically}.
	$$
	Then 
	$$
	\E \#\{\text{discarded}\}
	\ge \frac{n}{2 T + 1} - T.
	$$
\end{proposition}

We will prove this proposition in Section~\ref{s: lower} by showing that the reflected random walk on $[-T,T]$ achieves the smallest discard rate.

\section{Upper bound: proof of Theorem~\ref{thm: main}}	\label{s: main theorem proof}

Let $\pi$ be a probability density on $\R^d$ whose support lies in the interior of $K$. 
We use the Metropolis algorithm with the classically defined acceptance probability
\begin{equation}	\label{eq: axy}
	a(x,y) \coloneqq \frac{\pi(y)}{\pi(x)} \wedge 1,
\end{equation}
see e.g. \cite{Hastings}. Specifically, 
let us arrange the discards as follows:
\begin{itemize}
	\item Sample $w_0\sim\pi$.
	\item Upon receiving $\varepsilon_k v_k$, propose $w'=w_{k-1}+\e_k v_k$.
	\item With probability $a(w_{k-1},w')$, set $w_k = w'$; otherwise set $w_k = w_{k-1}$.
	\item Repeat for $\e_{k+1} v_{k+1}$.
\end{itemize}
\smallskip%

Since \(\varepsilon_k\) is a fair sign, the resulting Markov chain is reversible, and hence the distribution is preserved:
$$
w_k \sim \pi \text{ for all } k.
$$ 
Also, if \(w'\notin K\), then \(\pi(w')=0\), so the move is 
rejected. Hence, 
\(w_k \in K\) for all \(k\), and 
$$
w_k - w_0\in 2K \text{ for all } k,
$$
as claimed. Moreover, from the algorithm we see that
\begin{equation}	\label{eq: Ediscarded proof}
	\E\,\#\{\text{discarded}\}
	= \sum_{k=1}^n \Prob{\e_k v_k \text{ is discarded}}
	=\sum_{k=1}^n \E\bigl[1-a(w_{k-1},w_{k-1}+\e_k v_k) \bigr].
\end{equation}
We can compute the expectation above by conditioning on $\e_k v_k$, which is independent of 
$w_{k-1} \sim \pi$. Let's do this computation separately.

For a random vector $X \sim \pi$ and a fixed vector $v \in \R^d$, we have 
\begin{equation}	\label{eq: aXx+v}
	\E a(X,X+v)
	\overset{\eqref{eq: axy}}{=} \int_{\R^d} \left( \frac{\pi(x+v)}{\pi(x)} \wedge 1 \right) \pi(x) \, dx \\
	= 1 - \frac{1}{2} \int_{\R^d} \abs{\pi(x+v)-\pi(x)} \, dx,
\end{equation}
where the last equation follows from
$$
\left( \frac{\pi(x+v)}{\pi(x)} \wedge 1 \right) \pi(x)
= \pi(x+v) \wedge \pi(x)
= \frac{1}{2} \Big[ \pi(x+v) + \pi(x)
-\abs{\pi(x+v)-\pi(x)} \Big].
$$
Now, 
\begin{align*}
	\int_{\R^d} \abs{\pi(x+v)-\pi(x)} \, dx
	&= \int_{\R^d} \abs*{\int_0^1 \ip{\nabla \pi(x+tv)}{v} \, dt} dx
\le \int_0^1 \underbrace{\int_{\R^d} \abs*{\ip{\nabla \pi(x+tv)}{v}} \, dx}_{\text{does not depend on $t$}} \, dt \\
	&= \int_{\R^d} \abs*{\ip{\nabla \pi(x)}{v}} \, dx 
	= \E \abs*{\frac{\ip{\nabla \pi(X)}{v}}{\pi(X)}}
				\quad \text{(since $X \sim \pi$)}\\
	&= \E \abs*{\ip{\nabla \log \pi(X)}{v}}
	\le \sqrt{\E \ip{\nabla \log \pi(X)}{v}^2}\\
	&= \sqrt{v^\tran \I_\pi v}
		\quad \text{by definition of Fisher information matrix \eqref{eq: Ipi}.}
\end{align*}

Plugging this into \eqref{eq: aXx+v}, we get 
$$
\E \bigl[ 1 - a(X,X+v) \bigr]
\le \frac{1}{2} \sqrt{v^\tran \I_\pi v}.
$$
Finally, substitute this into \eqref{eq: Ediscarded proof} to obtain \eqref{eq: Ediscarded}. Theorem~\ref{thm: main} is proved. \qed

\section{Lower bound: proof of Proposition~\ref{prop: lower}}	\label{s: lower}

Without loss of generality, assume $T\ge 0$ is an integer, since partial sums are integers. Given a walk with steps $\e_1,\ldots,\e_n$ starting at some point $s \in [-T,T]$, consider all ways of discarding steps so that the resulting walk remains within $[-T,T]$:

\begin{definition}[$s$-valid subsequences]
	Fix a sign sequence $\e = (\e_1, \ldots, \e_n) \in \{-1,1\}^n$ and an integer \(s\in [-T,T] \). We say that an index subsequence $1 \le i_1 < \cdots < i_m \le n$ is {\em \(s\)-valid} if  
	\[
	\max_{1\le k\le m}\left| s + 
	\sum_{r=1}^k \e_{i_r}
	\right|\le T .
	\]
	Denote by $\ell(\e,s)$ be the maximum length of an $s$-valid index subsequence. 
\end{definition}

\begin{lemma}[The starting point does not matter much] \label{l0ls}
	\(\ell(\e,0) \le \ell(\e,s) + \abs{s}\).
\end{lemma}

\begin{proof}
Assume $s>0$ (the case $s<0$ is symmetric). Let $I$ be a longest $0$-valid subsequence. Scan $I$ left-to-right and delete indices with $\e_j=+1$ until $s$ such indices are deleted (or none remain). The new subsequence \(I'\) is \(s\)-valid; hence, $\ell(\e,0) \le \abs{I'}+s \le \ell(\e,s)+\abs{s}$.
\end{proof}

A simple algorithm to produce an \(s\)-valid subsequence is to run a {\em reflected walk} on $[-T,T]$ started at $s$: 
\begin{itemize}
	\item Initialize $S \coloneqq s$ and $I \coloneqq \emptyset$.
	\item For $k=1,2,\dots,n$:
		\begin{itemize}
			\item if $\abs{S+\e_k} \le T$, append \(k\) to \(I\) and update \(S:=S + \e_k\);
			\item otherwise, discard \(k\).
		\end{itemize}
\end{itemize}

Let us show that the reflected random walk outputs a longest $s$-valid subsequence. To show this, let \(i_1<\cdots<i_m\) and \(j_1<\cdots<j_m\) be two index subsequences of the same length. We define the lexicographic order by 
$$
(i_1,\dots,i_m)\prec_{\mathrm{lex}}(j_1,\dots,j_m) 
$$
if, at the first \(k\) with \(i_k\neq j_k\), we have \(i_k<j_k\).

\begin{lemma}[Optimality of the reflected walk]	\label{lem: optimality}
	Fix $s\in[-T,T]$ and $\e \in \{-1,1\}^n$. The reflected walk on $[-T,T]$ started at $s$ produces the lexicographically smallest among all longest $s$-valid subsequences.
\end{lemma}

\begin{proof}
Among all longest $s$-valid subsequences, let $I=(i_1<\cdots<i_m)$ be lexicographically smallest. Fix $1\le k\le n$ and set:
\[
S := s + \sum_{r:\, i_r<k} \e_{i_r}.
\]
It suffices to show the following implication: 
$$
\abs{S+\e_k} \le T
\quad \Longrightarrow \quad 
k \in I. 
$$
(Then, by induction, the reflected walk must produce $I$, proving the lemma.)

To show this, we use an exchange argument. Assume \(\e_k=+1\) (the other case is symmetric). Suppose for contradiction that \(I\) does not contain \(k\) while \(S \le T - 1\). 

Let \(j>k\) be the smallest index in \(I\) 
with \(\e_j=+1\), if it exists. If no such \(j\) exists, insert \(k\) into \(I\); the new sequence remains 
\(s\)-valid and is longer than \(I\). This contradicts that \(I\) has maximum length.

Otherwise, form \(I'\) by the index inserting 
\(k\) and deleting \(j\). Then \(I'\) is \(s\)-valid. It has the same length as \(I\) but is lexicographically smaller. This contradiction completes the proof.
\end{proof}

We are now ready to prove Proposition~\ref{prop: lower}. It suffices to consider the reflected walk, since Lemma~\ref{lem: optimality} shows that any other discarding algorithm must discard at least as many terms. 

From now on, let \(\e = (\e_1,\dots,\e_n) \) be a vector of independent Rademacher random variables. Let $S_0$ be an independent random variable uniformly distributed over the integers in $[-T,T]$. Let us run the reflected walk started at \(S_0\):
\[
S_k :=
\begin{cases}
S_{k-1}+\e_k, & \text{if } \abs{S_{k-1} + \e_k} \le T \quad(\text{accept}),\\
S_{k-1}, & \text{otherwise}\quad(\text{discard}).
\end{cases}
\]
A discard at time $k+1$ occurs iff $(S_k,\e_{k+1})=(T,+1)$ or $(-T,-1)$. Hence:
\[
\E \#\{\text{discards}\}
= \frac12\sum_{k=0}^{n-1}\big( \Prob{S_k=T} + \Prob{S_k=-T} \big) = \sum_{k=0}^{n-1} \Prob{S_k=T} 
= \frac{n}{2T+1}.
\]
To see the last equality, note that the reflected random walk is reversible, so the uniform distribution on $\{-T,\ldots,T\}$ is stationary for $S_k$, hence $\Prob{S_k=T} = 1/(2T+1)$. 

According to Lemma~\ref{lem: optimality}, the reflected walk above has length $\ell(S_0,\e)$, so we proved that
$$
\E \ell(S_0,\e) = n - \frac{n}{2T+1}.
$$
Finally, by Lemma~\ref{l0ls}, we have
\[
\E \ell(\e,0)\le 
\E \ell(\e, S_0)+ \E \abs{S_0}
\le n - \frac{n}{2T+1} + T,
\]
and the proof of Proposition~\ref{prop: lower} is complete.

 
 

\bibliographystyle{plain}
\bibliography{ref}

\end{document}